\documentclass{amsart}
\usepackage{amsfonts}
\usepackage{enumerate}
\usepackage{amsthm}
\usepackage{latexsym}
\usepackage{amsmath}

\theoremstyle{definition}
\newtheorem{defi}{Definition}[section]
\newtheorem{remark}{Remark}

\theoremstyle{plane}
\newtheorem{theo}{Theorem} 
\newtheorem{lem}{Lemma}[section]
{}
\newtheorem{proposition}{Proposition}[section]
\newtheorem{cor}{Corollary}[section]

\newcommand{\R}{{\mathbb R}}

\newcommand{\N}{{\mathbb N}}

\newcommand{\ha}{{\rm Harm}}

\newcommand{\dis}{\displaystyle}
\newcommand{\la}{\langle}
\newcommand{\ra}{\rangle}

\title{Spherical designs of harmonic index $t$}
\author{Eiichi Bannai}
\address{Department of Mathematics, Shanghai Jiao Tong University, 800 Dongchuan RD. Minhang District, Shanghai, China}
\email{bannai@sjtu.edu.cn, bannai@math.kyushu-u.ac.jp}
\author{Takayuki Okuda}
\address{Graduate School of Information Sciences, Tohoku University, Aramaki aza Aoba 6-3-09, Aoba-ku Sendai-city Miyagi-pref. 
980-8579, Japan}
\email{okuda@ims.is.tohoku.ac.jp}
\author{Makoto Tagami}
\address{Graduate School of Computer Science and Systems Engineering, Kyushu Institute of Technology,
	680-4 Kawazu, Iizuka-shi, Fukuoka, 820-8502, Japan}
\email{tagami@ces.kyutech.ac.jp}
\begin{document}
\maketitle

\begin{abstract}
Spherical $t$-design is a finite subset on sphere such that, for any polynomial of degree at most $t$, the average value of the integral on sphere can be replaced by the average value at the finite subset. It is well-known that an equivalent condition of spherical design is given in terms of harmonic polynomials. In this paper, we define a spherical design of harmonic index $t$ from the viewpoint of this equivalent condition, and we give its construction and a Fisher type lower bound on the cardinality. Also we investigate whether there is a spherical design of harmonic index attaining the bound.    
\end{abstract}
\section{Introduction}
Let $t$ be a natural number, $S^{n-1}$ the $(n-1)$-dimensional unit sphere centered at the origin. A finite nonempty subset $X$ on $S^{n-1}$ is called a \textit{spherical $t$-design} if, for any polynomial $f(x)=f(x_1,\ldots,x_n)$ of degree at most $t$, the following equality holds:
\[\frac{1}{|S^{n-1}|}\int_{S^{n-1}}f(x)d\sigma(x)=\frac{1}{|X|} \sum_{x\in X} f(x),\]
where $\sigma$ is an $O(\R^n)$-invariant measure on $S^{n-1}$ and $|S^{n-1}|$ denotes the surface volume of the sphere $S^{n-1}$.
The concept of spherical design was defined by Delsarte-Goethals-Seidel (refer to \cite{DGS,D,BB99,BB09}). A spherical $t$-design means to be a good configuration of points on sphere so that the average value of the integral of any polynomial of degree at most $t$ on sphere can be replaced by the average value at the finite set on sphere.

Let $\triangle=\frac{\partial^2}{\partial x_1^2}+\cdots+\frac{\partial^2}{\partial x_n^2}$ be the Laplacian.
A polynomial $f(x)$ is \textit{harmonic} if $\triangle f(x)=0$. Then put
\[
\ha_t(\R^n)=\left\{f(x) \mid \mbox{$f(x)$ is a harmonic and homogeneous polynomial of degree $t$ on $\R^n$} \right\}. 
\]
The dimension of $\ha_t(\R^n)$ is ${n+t-1 \choose t}-{n+t-3 \choose t-2}$ (refer to \cite[page 478 ]{AAR}). 
Some equivalent conditions of spherical design are known. In particular, the following condition is quite often used(\cite{D,BB99,BB09}):  
\begin{center}
For any $f(x) \in \ha_j(\R^d)$ and $1\leq j \leq t$, $\sum_{x\in X} f(x)=0$.
\end{center}
From this condition, we introduce the following notion which is the main subject in this paper:
\begin{defi}[Spherical design of harmonic index $t$]\label{index}
A finite nonempty subset $X$ on $S^{n-1}$ is called a \textit{spherical design of harmonic index $t$} (or simply, harmonic index $t$-design) if, for any $f(x)\in \ha_t(\R^n)$,  $\sum_{x\in X} f(x)=0 $.
\end{defi}
We are interested in what figure appears as spherical designs of harmonic index, and whether we can give a natural lower bound for harmonic index designs similar to the case of usual spherical design.

When $t$ is odd, any antipodal two points $X=\{x, -x\}$ forms a harmonic index $t$-design because $\sum_{x\in X} f(x)=f(x)+f(-x)=f(x)-f(x)=0$ for any $f(x)$ in $\ha_t(\R^n)$. So, from now on, we consider only the case when $t$ is even.

When $t$ is even, for any $f(x) \in \ha_t(\R^n)$, $f(-x)=f(x)$. So we remark that one can take harmonic index $t$-designs just on hemisphere. For any $n, t  \in \N$, from Seymour-Zaslavsky's theorem\cite{SZ}, if we make the number of vertices big enough, there always exists a harmonic index $t$-design on $S^{n-1}$. We denote the minimum cardinality of harmonic index $t$-design on $S^{n-1}$ by $A(n,t)$. From the above, we see that, when $t$ is odd, $A(n,t)=2$.

First we consider the case when $n=2$ and $t=2e$. Let $x,y$  be two unit vectors in $\R^2$ with angular $\theta=j\pi/2e$ for odd $j$. Then $X=\{x,y\}$ is a harmonic index $2e$-design on $S^1$. So we see that $A(2,t)=2$.

Next we consider the case when $t=2$ and $n \ge 2$. Let $X=\{e_1,\ldots,e_n\}$ be an orthonormal basis of $\R^n$ (that is, an antipodal half part of regular cross-polytope). Then it is easy to see that $X$ is a harmonic index $2$-design on $S^{n-1}$. Therefore $A(n,2)\le n$. In fact, we will show $A(n,2)=n$ later.
\footnote{We note that this is also proved as follows. If $X$ is a harmonic index $2$-design in $S^{n-1}$, then $X \cup (-X)$ is a spherical 3-design (possibly with multiplicities). The classification of tight spherical $3$-desigs(possibly with multiplicities) implies the claim.}

Let $C_t^\lambda(x)$ be the Gegenbauer poynomial of degree $t$, that is, when $\lambda\not=0$, 
\[(1-2xr+r^2)^{-\lambda}=\sum_{t=0}^\infty C_t^\lambda(x) r^t,\]
and when $\lambda=0$, it is  the Chebychev polynomial of first kind, that is, $C_t^0(x)=T_t(x)=\cos(t \arccos(x))$.
$\{C_t^\lambda(x)\}$ is a set of orthogonal polynomials on an interval $[-1,1]$ with a weight function $(1-x^2)^{\lambda-1/2}$ and so they have all roots on $[-1,1]$ (refer to \cite{AAR, E}). For simplicity, we normalize $C_t^\lambda(x)$ as $Q_{n,t}(x)=Q_t(x)=A_{n,t} C_t^{\frac{n-2}{2}}(x)$, where a normalization factor $A_{n,t}$ is given so as to be $Q_t(1)={n+t-1 \choose t}-{n+t-3 \choose t-2}$. For $x,y \in \R^n$, $\la x,y \ra$ denotes the standard inner product. Then (for covenience, using the same symbol $Q_{n,t}$), $Q_{n,t}(x,y)=Q_{n,t}(\la x,y \ra)$ becomes the reproducing kernel on the Hilbert space $\ha_t(S^{n-1})$ which consists of all functions of  $\ha_t(\R^n)$ restricted on $S^{n-1}$ (refer to \cite{HP}). Also for a finite subset $X$ on sphere, we set $I(X)=\{\la x,y\ra \mid x, \not= y  \in X\}$.

The following is the main theorem I in this paper, which gives a simple construction of harmonic index designs by using usual spherical designs in lower dimension by 1.  
\begin{theo}\label{const}
Let $X$ be a spherical $t$-design on $S^{n-2}$, $r$ a root of $Q_{n,t}(s)$. Set $X' \subset S^{n-1}$ to be :
\[X'=\left\{\left(r,\sqrt{1-r^2}x\right) \in S^{n-1}\mid x \in X\right\}.\]
Then $X'$ is a harmonic index $t$-design.
\end{theo}
The proof will be given in Section \ref{proof}. 
By Bondarenko--Radchenko--Viazovska's result in \cite{BRV,BRV2}, 
for a fixed $n$, 
there exists a spherical $t$-design on $S^{n-1}$ of size $O(t^{n-1})$. 
Hence Theorem \ref{const} gives $A(n,t) \leq O(t^{n-2})$ as a function of $t$
(see Corollary \ref{cor:BRV} in Section \ref{proof} for more details). 

Since the regular $(2e+1)$-gon is a spherical $(2e)$-design on $S^1$, by Theorem \ref{const}, taking the radius suitably and putting it on $S^2$, we obtain a harmonic index $(2e)$-design on $S^{2}$. To construct a harmonic index $4$-design by Theorem \ref{const}, we must take a regular 5-gon. By computer calculation, we will show that there does not exist a harmonic index 4-design of 4 vertices on $S^2$ and so $A(3,4)=5$. Moreover we will show that the configuration of points giving $A(3,4)=5$ is essentially unique, that is, they are given by the configuration obtained from a regular 5-gon using  Theorem \ref{const} (for these proofs, refer to Section \ref{comp}). Similarly, from a regular 7-gon, we have $A(3,6) \le 7$. Also up to the present, we have not detected a harmonic index $6$-design with at most 6 vertices. From these facts, at first, we conjectured that a minimum size harmonic index $(2e+1)$-design on $S^2$ is always given by a regular $(2e+1)$-gon and it is unique. But in fact, we realized from the following examples that, even if $t$ is big, there exists a harmonic index $t$-design with rather small number of vertices. For examples, take 6 vertices of an antipodal half part from 12 vertices of a regular icosahedron, then they form a harmonic index 8-design and moreover it is also a harmonic index 14-design. This fact follows from the harmonic Molien series of the icosahedral group. This number of vertices, 6 is far smaller than the number of vertices for hamonic index 14-design, 15, which is obtained from the regular 15-gon on $S^1$ by Theorem \ref{const}. A similar phenomenon can be seen in the following example, too. For example, in the case when $n=8$, the root system of type $E_8$, 240 points (and its antipodal half part, 120 points) is a harmonic index 10-design. While, the Fisher type inequality for usual spherical 10-designs on $S^6 \subset \R^7$ gives a lower bound $ {11\choose 5}+{10\choose 4}=672$, and it is far bigger than 120. Also in the case when $n=4$, 120 points of the $600$-cell (and its antipodal half part, 60 pointsjis a harmonic index 58-design. While, the Fisher type inequality for usual spherical 58-designs on $S^2 \subset \R^3$ gives a lower bound $ {31 \choose 29}+{30\choose 28}=900$, and it is far bigger than 60.

For spherical designs, there is the lower bound which is called the Fisher type inequality (refer to \cite{DGS}). The following is the main theorem II in this paper, which gives a Fisher type inequality for harmonic index designs.
\begin{theo}\label{linear}
Let $X$ be a harmonic index $t$-design on $S^{n-1}$. Put
\[c_{n,t}=- \min_{-1\le x \le 1} Q_{n,t}(x).\]
Then the following inequality holds:
\begin{equation}\label{fis}
|X| \ge 1+\frac{1}{c_{n,t}}\left\{{n+t-1\choose t}-{n+t-3\choose t-2}\right\}.
\end{equation}
Moreover, equality holds in (\ref{fis}) if and only if for any $a\in I(X)$, $c_{n,t}=-Q_{n,t}(a)$.
\end{theo}
We note that, since $Q_{n,t}(x)$ is an orthogonal polynomial on $[-1,1]$, $c_{n,t}$ is always positive. The proof of Theorem \ref{linear} will be given in Section \ref{proof}.
Denote the lower bound of (\ref{fis}) by $b_{n,t}$F
\[b_{n,t}:=1+\frac{1}{c_{n,t}}\left\{{n+t-1\choose t}-{n+t-3\choose t-2}\right\}.\]
We are interested in the case when $b_{n,t}$ is an integer. In this case, if there exists a harmonic index $t$-design $X$ on $S^{n-1}$ whose cardinality is exaclty $b_{n,t}$, then we say that $X$ is a \textit{tight} harmonic index $t$-design. For example, in the case when $t=2$, $b_{n,2}=n$, and, as we stated above, an antipodal half part of a regular cross-polytope is a harmonic index $2$-design with the cardinality $|X|=n$, and therefore it is a tight harmonic index $2$-design. In particular, we have $A(n,2)=n$.

A tight harmonic index design forms equiangular lines (refer to Corollary \ref{2dis} and for equiangular lines, refer to Brouwer-Haemers\cite[page 161]{BH}). Also fix $n$ and make $t$ big limitlessly, then $b_{n,t}$ converges some value, and, in general, the convergence value becomes bigger than the absolute bound $n(n+1)/2$ on equiangular lines on $\R^n$. Therefore we see that there hardly exists a tight harmonic index design for general $n,t$. On non-existence of tight harmonic index design, refer to Section \ref{tight}.

The contents of this paper is as follows. In Section \ref{proof}, we will give the proof of the main theorems I and II. In Section \ref{comp}, we will state some calculation results on harmonic index designs, in particular, we will state on a unique configuration of points giving $A(3,4)=5$. In Section \ref{tight}, we will give a necessary condition for parameters such that a tight harmonic index design exists, and we will show that tight harmonic index designs hardly exist. Also in Appendix I, we give a table on $b_{n,t}$ of Theorem \ref{linear} for $3\le n \le 10$ and $4\le t \le 20$. In Appendix II, we will give a Groebner basis used to show the uniqueness of points configuration giving $A(3,4)=5$.
\section{Proof of Main Theorems I,II}\label{proof}
\textit{Proof of Theorem \ref{const}}
We use the construction of an orthogonal basis in $\ha_{t}(S^{n-1})$ which is given in Andrews-Askey-Roy \cite[page 461]{AAR}. But here, for convenience, we use $Q_{n,t}(x)$ instead of $C_t^\lambda(x)$ in \cite{AAR}. Let $d_{n,t}:=\dim \ha_t(S^{n-1})$, and for $0\le j\le t$, 
$\{S_{j,l}(\xi_{n-1}) :  1\le l \le d_{n-1,j}\}$ be an orthogonal basis of $\ha_j(S^{n-2})$ where $\xi_{n-1} \in S^{n-2}$. Then if we take a variable $\xi \in S^{n-1}$ as $\xi=x/|x|=\left(s,\sqrt{1-s^2}\xi_{n-1}\right)$
\[\left\{S_{j,l}(\xi_{n-1})(1-s^2)^\frac{j}{2}Q_{2j+n,t-j}(s) : 0\le j \le t, 1\le l \le d_{j,n-1}\right\}\]
becomes an orthogonal basis of $\ha_{t}(S^{n-1})$. It is clear that the sum over $X$ for functions in the above basis of $\ha_{t}(S^{n-1})$ except for $j=0$ is 0 since $X$ is a spherical $t$-design on $S^{n-2}$. Also in the case when $j=0$, the function of the above basis is $Q_{n,t}(s)$. Therefore, if we take a root of $Q_{n,t}(s)$ as r, then all functions of the above basis satisfies that the sum over $X'$ is 0, and the theorem follows.
\hfill{$\Box$}
\begin{lem}\label{equ}
$X \subset S^{n-1}$ is a harmonic index $t$-design if and only if
\[\sum_{x,y \in X} Q_{n,t}(\la x,y \ra) = 0.\]
\end{lem}
\begin{proof}
Let $\left\{f_i(x)\mid 1\le i \le d_{n,t}\right\}$ is an orthonormal basis of a Hilbert space $\ha_t (S^{n-1})$. Then since $Q_{n,t}(x,y)$ is the reproducing kernel of $\ha_t (S^{n-1})$, we have the following addition formula:
\[Q_{n,t}(x,y)=Q_{n,t}(\la x,y \ra)=\sum_{i=1}^{d_{n,t}}f_i(x)f_i(y).\]
Hence,
\[\sum_{x,y \in X} Q_{n,t}(\la x,y \ra)=\sum_{x,y \in X}\sum_{i=1}^{d_{n,t}}f_i(x)f_i(y)=\sum_{i=1}^{d_{n,t}}\sum_{x,y \in X}f_i(x)f_i(y)=\sum_{i=1}^{d_{n,t}}\left(\sum_{x\in X} f_i(x)\right)^2.\]
From this equation, this lemma follows.
\end{proof}

By combining Theorem \ref{const} with results in \cite{BRV,BRV2} (see Remark \ref{rem:BRV} below for more details), 
we obtain the following corollary:

\begin{cor}\label{cor:BRV}
There exists a constant $C_{n-2}$ depending only on $n$ such that 
for each $N > C_{n-2} t^{n-2}$,
there exists a spherical design of harmonic index $t$ on $S^{n-1}$ of size $N$. 
\end{cor}

\begin{remark}\label{rem:BRV}
Let us fix $n$ and $t$.
In \cite{BRV}, Bondarenko--Radchenko--Viazovska proved that 
for each $N \geq C_{n-1} t^{n-1}$, 
there exists a sequence $x_1,\dots,x_N \in S^{n-1}$ such that 
\begin{align}
\frac{1}{N}\sum_{i=1}^{N} f(x_i) = \frac{1}{|S^{n-1}|} \int_{S^{n-1}} f(x) d\sigma(x) \label{defn:BRVdesign}
\end{align}
for any polynomial $f(x)$ of degree at most $t$,
where $C_{n-1}$ is a constant depending only on $n$.
It should be remarked that 
a sequence $x_1,\dots,x_N \in S^{n-1}$ with the property \eqref{defn:BRVdesign} is called a ``spherical $t$-design'' in their paper.
However, in our definition of spherical designs, a sequence $x_1,\dots,x_N$ is required to be distinct each other.
In \cite{BRV2}, Bondarenko--Radchenko--Viazovska improved their result as follows:
there exist positive constants $C_{n-1}'$ and $\lambda_{n-1}$ depending only on $n$
such that for each $N > C_{n-1}' t^{n-1}$, 
we can find a ``spherical $t$-design'' $x_1,\dots,x_N \in S^{n-1}$ in their sense with 
\[
|x_i - x_j| > \lambda_{n-1} N^{-1/{(n-1)}} \quad \text{for any distinct } i,j.
\]
Especially, $x_1,\dots,x_N \in S^{n-1}$ gives a spherical $t$-design of size $N$ even in our sense.
\end{remark}

\textit{Proof of Theorem \ref{linear}}
Let $F(s) := c_{n,t} + Q_{n,t}(s)$. Then from the definition of $c_{n,t}$, $F(s)$ is a non-negative function on $[-1,1]$. We evaluate $\sum_{x,y \in X} F(\la x,y \ra)$ in two ways.

Since $X$ is a harmonic index $t$-design, by Lemma \ref{equ}, $\sum_{x,y \in X} Q_{n,t}(\la x,y \ra) = 0$.
Therefore,
\begin{equation}
\sum_{x,y \in X} F(\la x,y \ra) =\sum_{x,y \in X}\left( c_{n,t} + Q_{n,t}(\la x,y \ra)\right)=c_{n,t} |X|^2 \label{eq:1}.
\end{equation}
Next, since $F(s)$ is non-negative on $[-1,1]$ and $Q_{n,t}(1) =\dim \ha_t(\R^n)={n+t-1 \choose t}-{n+t-3 \choose t-2}$, 
\begin{equation}\label{eq:2}
\sum_{x,y \in X} F(\la x,y \ra) \geq \sum_{x \in X} F(\la x,x \ra)=|X|F(1)=|X| \left\{c_{n,t} + {n+t-1 \choose t}-{n+t-3 \choose t-2}\right\}.
\end{equation}
By \eqref{eq:1} and \eqref{eq:2}, we have the following inequality: 
\[c_{n,t} |X|^2 \geq |X|\left\{c_{n,t}+ {n+t-1 \choose t}-{n+t-3 \choose t-2}\right\}
\] 
Therefore,
\[|X| \ge 1+\frac{1}{c_{n,t}}\left\{{n+t-1\choose t}-{n+t-3\choose t-2}\right\}.\]
Moreover, equality in this inequality holds if and only if equality holds in \eqref{eq:2}, that is, for any $x,\not = y \in X$, $c_{n,t} + Q_{n,t}(\la x,y \ra)=0$.  Therefore the final assertion of the theorem follows.
\hfill{$\Box$}
\section{Some calculation results}\label{comp}
First we apply Theorem \ref{linear} to the case when $n=2$. Since $Q_{2,t}(x)=2\cos(t\arccos(x))$,
\[c_{2,t}=- \min_{-1\le x \le 1} 2\cos(t\arccos(x))=2.\]
Therefore, we have
\[b_{2,t}=1+\frac{1}{c_{2,t}}\left\{{t+1 \choose t}-{t-1\choose t-2}\right\}=2,\]
and so $A(2,t)\ge 2$. Let $x,y$ be two unit vectors with angular $\theta=j\pi/2e$ where $j$ is odd. Then $X=\{x,y\}$ is a harmonic index $(2e)$-design on $S^1$ and we have $A(2,t)=2$. Conversely let $X=\{x,y\} \subset S^1$ be a harmonic index $t$-design with two vertices. From the final assertion of Theorem \ref{linear}, $\cos t\la x,y\ra=-1$ must hold. Hence we see that $X$ must be obtained by the above construction.

Next we consider the case when $t=2$. Then since $Q_{n,2}(x)=\dis \frac{n+2}{2}(nx^2-1)$,
\[c_{2,t}=- \min_{-1\le x \le 1}Q_{n,2}(x)=\frac{n+2}{2}.\]
This value is given at $x=0$. Therefore, we have
\[b_{n,2}=1+\frac{1}{c_{n,2}}\left\{{n+1\choose 2}-{n-1\choose 0}\right\}=n,\]
and $A_{n,2}\ge n$. Let $X=\{e_i\mid 1\le i \le n\}$ be an orthonomal basis of $\R_n$. Then $X$ is a harmonic index $2$-design on $S^{n-1}$ and so we have $A(n,2)=n$. Conversely let $X$ be a harmonic index $2$-design on $S^{n-1}$ with $n$ vertices. Then From the final assertion of Theorem \ref{linear}, we see that, for any $x \not= y \in X$, $\la x,y\ra=0$. So $X$ must be an orthonormal basis of $\R_n$.

The first non-trivial case is the case when $n=3$ and $t=4$.
\begin{theo}\label{unique}
$A(3,4)=5$. Moreover a harmonic index 4-design with 5 vertices on $S^2$ is congruent to the following $X_0^{\pm}$ or any one of the configurations obtained by replacing some points in $X_0^{\pm}$ to its antipodal points:
\begin{eqnarray*}
X_0^{\pm}=&&\left\{\left(\frac{1}{35}\sqrt{525\pm70\sqrt{30}},\frac{1}{35}\sqrt{700\mp70\sqrt{30}},0\right),\right.\\
&&\left(\frac{1}{35}\sqrt{525\pm 70\sqrt{30}},\frac{\sqrt{5}-1}{140}\sqrt{700\mp70\sqrt{30}},\frac {\sqrt{10+2\sqrt{5}}}{140}\sqrt{700\mp70\sqrt{30}}\right),\\
&&\left(\frac{1}{35}\sqrt{525\pm70\sqrt{30}},-\frac{\sqrt{5}+1}{140}\sqrt{700\mp70\sqrt{30}},\frac {\sqrt{10-2\sqrt{5}}}{140}\sqrt{700\mp70\sqrt{30}}\right),\\
&&\left(\frac{1}{35}\sqrt{525\pm70\sqrt{30}},-\frac{\sqrt{5}+1}{140}\sqrt{700\mp 70\sqrt{30}},-\frac {\sqrt{10-2\sqrt{5}}}{140}\sqrt{700\mp70\sqrt{30}}\right),\\
&&\left.\left(\frac{1}{35}\sqrt{525\pm 70\sqrt{30}},\frac{\sqrt{5}-1}{140}\sqrt{700\mp 70\sqrt{30}},-\frac {\sqrt{10+2\sqrt{5}}}{140}\sqrt{700\mp 70\sqrt{30}}\right)\right\},
\end{eqnarray*}
Here double-sign is corresponding.
\end{theo}
\begin{proof}
First we apply Theorem \ref{linear}. Since
\[\dis Q_{3,4}(x)=\frac{9}{8}(35x^4-30x^2+3)=\frac{9}{8}\left\{35\left(x^2-\frac{3}{7}\right)^2-\frac{24}{7}\right\},\]
$c_{3,4}=27/7$. So $b_{3,4}=10/3=3.333\ldots$ and we have $A(3,4)\ge 4$. Next in order to give an upper bound of $A(3,4)$, we apply Theorem \ref{const}. A regular 5-gon is a spherical $4$-design on $S^1$ (cf. \cite{BB99, Hong}).
Hence by Theorem \ref{const}, adjusting the radius of a regular 5-gon suitably and putting it on $S^2$, we obtain a harmonic index  $4$-design with 5 vertices on $S^2$. These two configurations obtained in this way by using two positive roots of $Q_{3,4}(s)$ are $X_0^+$ and $X_0^-$ in the assertion of the theorem. Thus we have $A(3,4)\le 5$. A problem is whether there exists a harmonic index $4$-design with 4 vertices on $S^2$. We can solve this problem by the direct calculation as follows.

Take the following basis of $\ha_4(\R^3)$.
\begin{eqnarray*}
H=&&\{x^3y-xy^3,x^3z-3xy^2z,3x^2yz-y^3z,x^4-6x^2y^2+y^4,4xz^3-3x^3z-xy^2z,\\
&&4yz^3-3x^2yz-3y^3z,6xyz^2-x^3y-xy^3,6x^2z^2-x^4-6y^2z^2+y^4,\\
&& 8z^4-24x^2z^2-24y^2z^2+3x^4+6x^2y^2+3y^4\}.
\end{eqnarray*}
Let $X$ be a four point subset on $S^2$. Since the property as harmonic index design is invariant under orthogonal transformations, by Lemma \ref{equ}, without loss of generality, we may put the points of $X$ as follows:
\[X=\{x_1=(1,0,0),x_2=(s_{21},s_{22},0),x_3=(s_{31},s_{32},s_{33}),x_4=(s_{41},s_{42},s_{43})\}.\]
Here $s_{21},s_{22},\ldots,s_{43}$ are variables and satisfy $\la x_i,x_i\ra=1$ for $i=1,\ldots, 4$. Let 
\[EQ_1=\{\la x_i,x_i\ra-1 \mid i=1,\ldots,4\}, \;EQ_2=\left\{\sum_{x\in X}f(x) \mid f \in H\right\}.\] 
By Definiton \ref{index}, $X$ is a harmonic index $4$-design if and only if, for any $f \in H$, $\sum_{x\in X}f(x)=0$. Thus, the common zeros of $EQ:=EQ_1 \cup EQ_2$ give exactly harmonic index $4$-designs. We calculated the ideal which is generated by $EQ$ in a multivariate polynomial ring $\R[s_{21},s_{22},\ldots,s_{43}]$ by Groebner bases function in a computational algebra system, Magma, and we found that the ideal equals to the whole of the ring. It means that there is no common zero of $EQ$ and therefore, that there does not exist a harmonic index $4$-design with 4 vertices on $S^2$.

Next we show the uniqueness of configurations giving $A(3,4)=5$. Suppose $t$ is even. Here we note that, for a harmonic index $t$-design, even if we replace some points of them to its antipodal points, it also becomes a hamonic index $t$-design. 

In order to prove the uniqueness, first we count the number of the configurations given in the assertion of the theorem. 
As a promise to count them, we put $X=\{x_1,x_2,\ldots,x_5\}$, and two points in its five points are put as $x_1=(1,0,0),x_2=(a,b,0)$, and we count the number of values $a,b$ and the coordinates of $x_3,x_4,x_5$ such that $X$ is congruent to any one of the configurations given in the assertion of the theorem. The positive roots of $Q_{3,4}(s)$ are $r_1=\dis \frac{1}{35}\sqrt{525+70\sqrt{30}}$ and $\dis r_2=\frac{1}{35}\sqrt{525-70\sqrt{30}}$. $X_0^{\pm}$ in the theorem are exactly ones obtained by Theorem \ref{const} with $r_1$ and $r_2$, respectively, and they are cited as the small 5-gon and the big 5-gon on $S^2$, respectively. First, on $X_0^{\pm}$ and the configurations obtained by replacing some points in $X_0^{\pm}$ to its antipodal points (which are said to be the derived ones by this operation), we can confirm by numerical calculation that exactly 8 values appear among $I(X)$'s for the configurations $X$'s which are given in the assertion of the theorem. The 8 values of inner products are given by the following pairs of two
  points: two vertices neighboring on the small 5-gon, two vertices on the diagonals of the small 5-gon, a vertex of the small 5-gon and the antipodal points of the neighbor vetices on the 5-gon, a vertex of the small 5-gon and the antipodal points of the diagonal vertices on the 5-gon, and 4 kinds of the corresponding ones for the big 5-gon. The number of them is 8 in all.  
In particular, inner products appearing in these pairs of two points are all different. This fact means that, once we determine a distance between a pair of two points, it is determined which configuration the whole five points are on , the small 5-gon and the derived one, or the big 5-gon and the derived one. 

First fix $x_1=(1,0,0)$. For each of the above 8 distances, there are two choices to put $x_2$ by the distance from $x_1$ to a clockwise direction or a counterclockwise direction on the equator. Hence there are 16 choices as a position of $x_2$ overall. From the above fact, by the distance $x_1x_2$, it is determined which configuration the whole five points are on, the small 5-gon and the derived one, or the big 5-gon and the derived one. Hence after determining the distance $x_1x_2$, there are 6 positions to put $x_3$ among the remaining vertices of 5-gons and the derived one, and the positions all have different combinations of distances from $x_1$ and $x_2$. For each combination of distances, we need to choose which $x_3$ is on, the north hemisphere or the south hemisphere, so we have two choices. Finally, we have 12 choices as positions of $x_3$. Similarly, after determining the coordinates of $x_1,x_2$ and $x_3$, we see that there are 4 combinations of distances $x_1x_4,
 x_2x_4,x_3x_4$. Since a coordinate of a point in $\R^3$ is determined by distances from linearly independent three points, a coordinate of $x_4$ is determined by a combination of distances. Similarly, after determining the coordinates of $x_1,\ldots,x_4$, there are 2 positions to put $x_5$. After all, we see that there are $16\times 12 \times 4 \times 2=1536$ combinations for coordinates of $x_2,\ldots,x_5$ overall.

Next, we carry out a Groebner basis calculation for the case of five points in a similar way to the case of four points. Let $H$ be the same as the above and 
\[X=\{x_1=(1,0,0),x_2=(s_1,s_2,s_3),x_3=(s_4,s_5,s_6),x_4=(s_7,s_8,s_9),x_5=(s_{10},s_{11},s_{12})\},\]
where $s_3=0$. Also let 
\[EQ_1'=\{\la x_i,x_i\ra-1 \mid i=1,\ldots,5\},\; EQ_2'=\left\{\sum_{x\in X}f(x) \mid f \in H\right\}.\] 
Similar to the case of four points, the common zeros of $EQ':=EQ_1' \cup EQ_2'$ give exactly harmonic index $4$-designs.
For the ideal generated by $EQ'$ in a multivariate polynomial ring $\R[s_1,s_2,s_4\ldots,s_{12}]$, we calculate a Groebner basis of a lexicographical order using Magma. the calculation result is given in Appendix 2.

First we factorize $P_{16}$ as follows:
\[P_{16}=\frac{1}{866761}(931s_{12}^8-1302s_{12}^6+627s_{12}^4-126s_{12}^2+9)(931s_{12}^8-1428s_{12}^6+732s_{12}^4-144s_{12}^2+9).\]
Put
\begin{eqnarray*}
f_1=931s_{12}^8-1302s_{12}^6+627s_{12}^4-126s_{12}^2+9,\\
f_2=931s_{12}^8-1428s_{12}^6+732s_{12}^4-144s_{12}^2+9.
\end{eqnarray*}
We checked by Sturm's theorem that $f_1$ and $f_2$ have the eight real zeros.
First for eight zeros of $f_1$ in $s_{12}$, we investigate possible values for $s_{1},\ldots,s_{11}$.
Since $P_{15}$ is divided  by $f_1$, we do not need to consider $P_{15}$ for zeros of $f_1$. For each value of $s_{12}$, the value of $s_{11}$ is determined by a polynomial $P_{14}$ in $s_{11}$, and there are at most 4 possible values. Similarly, for each possibility of $s_{12}$ and $s_{11}$, $s_{10}$ has at most 2 possibilities by $P_{13}$. We are to determine the number of possibilities for $s_{1},\ldots,s_{11}$ in this way. $P_{12}$ is divided by $f_1$. 
$s_9$ has at most 4 possibilities by $P_{11}$. $s_8$ has at most a possibility by $P_{10}$, $s_7$ has at most a possibility by $P_6$, $s_6$ has at most 2 possibilities by $P_5$, $s_5$ has at most a possibility by $P_4$, $s_4$ has at most a possibility by $P_3$, $s_2$ has at most 2 possibilities by $P_2$. $s_1$ has at most a possibility by $P_1$. Finally, from the zeros of $f_1$, the total number of possibilities for $s_{1},\ldots,s_{11}$ is $8\times4^2\times 2^3$. Similarly we count the number of possibilities in the case of 8 zeros of $f_2$. Since $P_{15}$ is not divided by $f_2$, for each zero of $f_2$, $s_{11}$ has at most 2 possibilities by $P_{15}$. $s_{10}$ has at most 2 possibilities by $P_{13}$, $s_9$ has at most 2 possibilities by $P_{12}$. Since $P_8$, $P_9$, $P_{10}$ and $P_{11}$ are in the ideal generated by $P_{12}$, $P_{13}$, $P_{14}$, $P_{15}$ and $P_{16}$, $s_8$ has at most 2 possibilities by $P_7$. $s_7$ has at most a possibility by $P_6$, $s_6$ has at most 2
  possibilities by $P_5$, $s_5$ has at most a possibility by $P_4$, $s_4$ has at most a possibility by $P_3$, $s_2$ has at most 2 possibilities by $P_{2}$. $s_1$ has at most a possibility by $P_1$. Finally, for the zeros of $f_2$, the total number of possibilities is $8\times 2^6$. In conclusion, the number of the common zeros of $EQ'$ is at most $2^3\times4^2\times + 8\times 2^6=2^9\times3$, which is equal to the number of examples given in the assertion of the theorem. This completes the proof of the theorem.
\end{proof}

From the Groebner basis calculation, which is used in the proof of Theorem \ref{unique}, and the existence of an antipodal half of vertices of icosahedron, we see $5\le A(3,8)\le 6$. Other exact values of $A(n,t)$ are all open.
\section{On existence of tight harmonic index design}\label{tight}
In this section, we investigate the conditions for the existence of tight harmonic index designs. First we show the following:
\begin{cor}\label{2dis}
Suppose that  $b_{n,t}$ is a natural number and $X$ is a harmonic index $t$-design on $S^{n-1}$ with $|X|=b_{n,t}$.
Then there exists some $\alpha \in [-1,1]$ such that $I(X)=\{\pm \alpha\}$.
\end{cor}
\begin{proof}
By the final assertion of Theorem \ref{linear}, for any $\alpha \in I(X)$, $Q_{n,t}(\alpha)$ must attain the minimum of $Q_{n,t}(x)$ on $[-1,1]$. But it is known that local minima of Gegenbauer polynomials change monotonously from the origin (cf. Szeg\"o \cite[168 page]{Szego}). Hence the minimum point $\alpha$ is unique. Therefore the assertion of the corollary is concluded. 
\end{proof}
Let $J_\alpha(z)$ is the Bessel function of the fist kind for parameter $\alpha$. Refer to \cite{AAR,Szego} on Bessel funtion. $j_{\alpha,k}$ denotes the $k$-th positive root of $J_\alpha(z)$.
\begin{proposition}\label{shu}
Let $\dis F_n(z)=\left(\frac{z}{2}\right)^{-\frac{n-3}{2}}J_{\frac{n-3}{2}}(z)$. Then fix $n$ and make $t$ tend to infinity, then
\[\dis b_{n,t} \to 1-\frac{1}{F_n\left(j_{\frac{n-1}{2},1}\right)}.\]
\end{proposition}
\begin{proof}
Let $\dis \alpha=\frac{n-3}{2}$ and $P_{n,t}(s)=B_{n,t}Q_{n,t}(s)$. Here the normalization factor $B_{n,t}$ is taken so as to be $P_{n,t}(1)={t+\alpha \choose t}$. Then under the notation in Szeg\"o\cite{Szego}, $P_{n,t}(s)=P_{t}^{(\alpha,\alpha)}$. Also denote by $c_{n,t}'$ the minimum of $P_{n,t}(s)$ on $[-1,1]$ times -1. Then
\[\frac{1}{c_{n,t}}\left\{{n+t-1\choose t}-{n+t-3\choose t-2}\right\}=\frac{B_{n,t}}{B_{n,t}c_{n,t}}\left\{{n+t-1\choose t}-{n+t-3\choose t-2}\right\}=\frac{1}{c_{n,t}'}{t+\alpha \choose t}=\frac{P_{n,t}(1)}{c_{n,t}'}.\]
Since $\dis P_{n,t}(1)={t+\alpha \choose t}\sim t^\alpha$, $\dis \lim_{t\to \infty} b_{n,t}=\lim_{t\to \infty} 1-\frac{t^\alpha}{c_{n,t}'}$. By \cite[page 63]{Szego},
\[\frac{d}{ds}P_{n,t}(s)=\frac{1}{2}\left(t+n-2\right)P_{n+2,t-1}(s).\]
Since local minima of $P_{n,t}(s)$ change monotonously from the origin, the minimum is attained at the maximum root of $P_{n+2,t-1}(s)$. Let $\dis -1<x_{n,t}^{(t)}<x_{n,2}^{(t)}<\cdots<x_{n,1}^{(t)}<1$ be the roots of $P_{n,t}(s)$, and set $\dis x_{n,i}^{(t)}=\cos\theta_{n,i}^{(t)} \;\left(0<\theta_{n,i}^{(t)}<\pi\right)$. So the minimum of $P_{n,t}(s)$ is  $\dis c_{n,t}'=P_{n,t}\left(x_{n+2,1}^{(t-1)}\right)=P_{n,t}\left(\cos \theta_{n+2,1}^{(t-1)}\right)$. The following convergence is well-known (\cite[192 page]{Szego}):
\begin{equation}\label{limit}
\lim_{t \to \infty} t^{-\alpha}P_{n,t}\left(\cos \frac{z}{t}\right)=\left(\frac{z}{2}\right)^{-\alpha}J_{\alpha}(z).
\end{equation}
This convergence is uniform in every bounded region of the complex $z$-plane. Also by \cite[192 page]{Szego}, $\lim_{t\to \infty}t\theta_{n,i}^{(t)}=j_{\alpha,i}$. Finally, 
\begin{equation}\label{deno}
\dis \lim_{t\to \infty} b_{n,t}=\lim_{t\to \infty} 1-\frac{t^\alpha}{c_{n,t}'}=1-\lim_{t\to \infty} \frac{t^\alpha}{c_{n,t}'}=1-\lim_{t\to \infty} \frac{t^\alpha}{P_{n,t}\left(\cos \theta_{n+2,1}^{(t-1)}\right)}.
\end{equation}
Put $z_m=m\theta_{n+2,1}^{(m-1)}$. Then 
\[\lim_{m \to \infty}z_m=\lim_{m \to \infty}\frac{m}{m-1}(m-1)\theta_{n+2,1}^{(m-1)}=\lim_{m \to \infty}\frac{m}{m-1}\lim_{m \to \infty}(m-1)\theta_{n+2,1}^{(m-1)}=1\cdot j_{\alpha+1,1}=j_{\alpha+1,1}.\]
Since the convergence is uniform, 
\begin{eqnarray}
\lim_{t\to \infty}t^{-\alpha}P_{n,t}\left(\cos \theta_{n+2,1}^{(t-1)}\right)&=&\lim_{t\to \infty}t^{-\alpha}P_{n,t}\left(\cos \frac{t\theta_{n+2,1}^{(t-1)}}{t}\right)=\lim_{m \to \infty}\lim_{t\to \infty} t^{-\alpha}P_{n,t}\left(\cos \frac{z_m}{t}\right) \nonumber\\
&=& \lim_{m \to \infty}\left(\frac{z_m}{2}\right)^{-\alpha}J_{\alpha}(z_m)=F_n(j_{\alpha+1,1}). \label{saigo}
\end{eqnarray}
By (\ref{deno}) and (\ref{saigo}), the proof of the proposition is completed. 
\end{proof}
The set of lines passing through the origin in $\R^n$ is called \textit{equiangular lines} if any distinct two lines in the set make the same angle. If $X \subset \R^n$ satisfies $I(X)=\{\pm \alpha\}$, the set of lines combining the origin and the points of $X$ forms equiangular lines. Since $x$ and $-x$ give the same line, we obtain equiangular lines with at least $|X|/2$ lines from this construction. If a tight harmonic index $t$-design $X$ exists, then $I(X)=\{\pm \alpha\}$ by Corollary \ref{2dis}, and since $X=-X$ does not hold by the tightness, we obtain equiangular lines with at least $|X|/2+1$ lines. It is well-known that the cardinality of equiangular lines in $\R^n$ is bounded above by $n(n+1)/2$ (\cite{Lemmens}). For small $n$'s, calculating the covergence value of Proposition \ref{shu}, we have $b_3=3.482871935 \;(6)$, $b_4=5.079602836 \;(10)$, $b_5=8.559751097\;(15)$, $b_6=16.42679115\;(21)$, $b_7=35.11842602\;(28)$, $b_8=81.85047703\;(36)$, $b_9=204.5294426\;(45)
 $, $b_{10}=541.6547218\;(55)$ where the parentheses denotes the value $n(n+1)/2$ for each $n$. Hence we see that, for $n=8,9,10$, a tight harmonic index $t$-design does not exist for large enough $t$.
Also we note that $b_{n,t}$ is not monotonously increasing in the case when $n=4$.

From now on, we consider the case when $t=4$.
\begin{proposition}\label{4tight}
\[b_{n,4}=\frac{(n+1)(n+2)}{6}.\]
In particular, $b_{n,4}$ is an integer if and only if $n$ is not divided by 3. Furthermore, in that case, for $X \subset S^{n-1}$ with $|X|=b_{n,4}$, the following are equivalent:
\begin{enumerate}[(i)]
\item $X$ is a hamonic index $4$-design,
\item $I(X) \subset \left\{\pm \sqrt{\frac{3}{n+4}}\right\}$.
\end{enumerate}
\end{proposition}
\begin{proof}
\[Q_{n,4}(x)=\frac{n(n+6)}{24}\{(n+2)(n+4)x^4-6(n+2)x^2+3\}=\frac{n(n+6)}{24} \left\{  (n+2)(n+4)\left(s^2 -\frac{3}{n+4}\right)^2- \frac{6(n+1)}{n+4} \right \}.\]
Hence on $[-1,1]$, $Q_{n,4}(x)$ attains the minimum $c_{n,4}=\dis -\frac{n(n+1)(n+6)}{4(n+4)}$ at $\dis x^2=\frac{3}{n+4}$. Therefore, we have $\dis b_{n,4}=\frac{(n+1)(n+2)}{6}$. It is clear that $b_{n,4}$ is an integer if and only if $n$ is not a multiple of 3. Also the latter of the assertion in the proposition follows from the final assertion of Theorem \ref{linear}.
\end{proof}
\begin{lem}[Musin\cite{Musin}]\label{mu}
Let $N$ be a natural number and $0< \alpha< 1$. Then the following are equivalent:
\begin{enumerate}[(i)]
\item there exists $X \subset S^{n-1}$ with $|X|=N$ and $I(X)\subset \{\pm \alpha\}$,
\item there exists $Y \subset S^{n-2}$ with $|Y|=N-1$ and $\dis I(Y)\subset \left\{\frac{\alpha}{1+\alpha},\frac{-\alpha}{1-\alpha} \right\}$. Moreover if  $1/2<\alpha<1$, then we may suppose $\dis I(Y)\subset \left\{\frac{\alpha}{1+\alpha}\right\}$.
\end{enumerate}
\end{lem}
By Lemma \ref{mu}, we have the following corollary:
\begin{cor}\label{445}
For $t=4$ and $n=4,5$, there does not exist a tight harmonic index 4-design.
\end{cor}
\begin{proof}
We show the case when $t=4$ and $n=4$. The case when $n=5$ is shown similarly. Suppose that $X$ is s a tight $4$-design  on $S^3$. Then $X$ satisfies $|X|=5$ and $I(X)\subset \left\{\pm \sqrt{3/8}\right\}$. By Lemma \ref{mu}, in this case, there exists $Y \subset S^2$ such that $|Y|=4$ and
\[I(Y) \subset \left\{\frac{\sqrt{\frac{3}{8}}}{1+\sqrt{\frac{3}{8}}}\right\}=\left\{\frac{2\sqrt{6}-3}{5}\right\}.\]
But $Y$ is a regular simplex on $S^2$ with $I(Y)=\{-1/3\}$. This is a contradiction.
\end{proof}
In order to show the non-existence of tight harmonic index 4-designs on $S^{n-1}$ for some more $n$, we use the method of Einhorn-Schoenberg \cite{ES}. The finite subset $X$ on Euclidean space is called a 2-\textit{distance set} if the number of distances appearing on $X$ is exactly 2. Let $X=\{x_1,\ldots,x_m\}$ be a 2-distance set with distances $\{1,(<)b\}$. Construct a graph $G=(X,E)$ for $X$ as follows. The vertex set is $X$, and the edge is joined when and only when the distance is $b$. Let $B$ be the adjacency matrix of $G$ indexed by $X=\{x_1,\ldots,x_m\}$. Put $C:=(b^2-1)B+J-I$. Let $L$ be a $(m-1)\times (m-1)$ matrix with the $(i-1,j-1)$-entry $L_{i-1,j-1}:=C_{1i}+C_{1j}-C_{ij}$ where $C_{ij}$ denotes the $(i,j)$-entry of $C$ and $i,j$ move from 2 to $m$. Then if $X$ can be isometrically embedded in $\R^n$, then the rank of $L$ must be at most $n$.

Let $X$ be a tight harmonic index 4-design on $S^6$. Then by Proposition \ref{4tight}, $\dis |X|=12, I(X)\subset \{\pm \sqrt{3/11}\}$. In particular, there must be a 2-distance set of 9 points with $\dis b^2=\frac{7+\sqrt{33}}{4}$ on $\R^7$. By using Magma and the database of graphs with 9 points (including inconnected ones), we checked by the above method whether there exists a graph of 9 points which can be isometrically embedded in $\R^7$ as a 2-distance set with $\dis b^2=\frac{7+\sqrt{33}}{4}$, and from this calculation, we found 60 possibilities of graphs among them. Next considering all extensions of the 60 possibilities to graphs of 10 points, we checked by the above method whether there exists an extended graph which can be isometrically embedded in $\R^7$ with $\dis b^2=\frac{7+\sqrt{33}}{4}$ among them, and we found that such an extended graph of 10 points does not exist. This means that there does not exist a 2-distance set of 10 points in $\R^7$ with $\dis b^2=\frac{7+\sqrt{33}}{4}
 $ and in particular, there does not exist a tight harmonic index 4-design on $S^6$.

Similarly let $X$ be a tight harmonic index 4-design on $S^7$, then $|X|=15, I(X) \subset \{\pm 1/2\}$. Particularly, there is a 2-distance set of 10 points with $b^2=3$ in $\R^8$. Running through all graphs of 10 points, we investigated whether there is a graph whose $L$ is of rank at most 8, and we found that there is not such a graph. This means that there is not a 2-distance set of 10 points with $b^2=3$ in $\R^8$. Therefore we see that there is not a tight harmonic index 4-design on $S^7$.

Next we consider the case when $n=10$. In this case, let $X$ be a tight harmonic index 4-design in $\R^{10}$, then $|X|=22$, $I(X)=\left\{\pm \sqrt{\frac{3}{14}}\right\}$. By Lemma \ref{mu}, there is some $X' \subset S^8$ with $|X'|=21$, $\dis I(X')=\left\{\alpha=\frac{\sqrt{3}}{\sqrt{14}+\sqrt{3}},\beta=\frac{-\sqrt{3}}{\sqrt{14}-\sqrt{3}}\right\}$. Fix a point of $X'$ and set $X_\alpha$ and $X_\beta$ to be subsets of $X'$ with inner products $\alpha$ and $\beta$ from the point, respectively. Each of them lies on the circle with the same latitude. By Pigeonhole principle, $|X_\alpha|\ge 10$ or $|X_\beta|\ge 10$. In particular, considering like as $X_\alpha, X_\beta \subset S^7$ and adjusting the norm, we can think $X_\alpha$ and $X_\beta$ 2-distance sets with inner products $\dis \left\{\frac{\alpha}{1+\alpha}, \frac{\beta-\alpha^2}{1-\alpha^2}\right\}$, $\dis \left\{\frac{\beta}{1+\beta},\frac{\alpha-\beta^2}{1-\beta^2} \right\}$ on $S^7$. Similar to the above, by Einhorn-S
 choenberg's method, we checked  for all graphs of 10 points whether there is a 2-distance set with the above inner product in $\R^8$, and we found that there is not such a graph. Therefore, finally, we see that there is not a tight harmonic index 4-design in $\R^{10}$.

Finally we use the following theorem to investigate the existence of tight harmonic index 4-designs.
\begin{theo}[Larman-Rogers-Seidel\cite{LRS}]\label{larman}
Let $X\subset \R^n$ be a 2-distance set with distances $\alpha, (<) \beta$. If $|X|>2n+3$, then there is a natural number $k$ at least 2 such that $\alpha^2/\beta^2=(k-1)/k$.
\end{theo}
From this theorem, we have the following:
\begin{theo}
Let $X$ be a tight harmonic index 4-design on $S^{n-1}$. Then $n=2$ or there is an odd $p\ge 5$ such that $n=3p^2-4$.
\end{theo}
\begin{proof}
Suppose $n\le 10$. By Proposition \ref{4tight}, when $b_{n,4}$ is an integer, $n=2,4,5,7,8,10$. But as we mentioned above, in these cases, there is not a tight harmonic index 4-design. Therefore we suppose $n\ge 11$, and then $b_{n,4}\ge 2n+4$. By Proposition \ref{4tight}, if $X$ is a tight harmonic index 4-design on $S^{n-1}$, then $X$ is a 2-distance set with $I(X)=\{\pm \sqrt{3/(n+4)}\}$. Therefore, by Theorem \ref{larman}, there is a natural number $k$ such that 
\[\frac{2-2\sqrt{\frac{3}{n+4}}}{2+2\sqrt{\frac{3}{n+4}}}=\frac{k-1}{k}.\]
Deforming this equation, we have
\[(2k-1)\sqrt{\frac{3}{n+4}}=1.\]
Hence if we put $p=2k-1$, then $n=3p^2-4$. 

When $p=3$, that is, $n=23$, a tight harmonic index 4-design $X$ in $\R^{23}$ satisfies $|X|=100$ and $I(X)=\left\{\pm \frac{1}{3}\right\}$. So we obtain equiangular lines of at least 51 lines in $\R^{23}$. But, since the maximum number of equiangular lines with inner product $1/3$ is 44 (\cite{Lemmens}), there is not a tight harmonic index 4-design in $\R^{23}$.
\end{proof}
\begin{remark} 
Wei-Hsuan Yu (University of Maryland) informed us
in a private communication (July, 2013) that there are no tight
harmonic 4-designs for $p=5, 7, 9$ in Theorem 5.
He showed that in these 3 cases the maximum sizes of
equiangular lines with inner product $1/p$
are respectively 416, 1506, 3952, by using the semi-definite
programming, which are strictly smaller than $(n+1)(n+2)/6=876, 3480, 9640$ (respectively), 
and this is a contradiction. Similar results are also expected for larger $p$, but
it seems difficult to deal with the semidefinite programming for
larger dimensions.
\end{remark}
\paragraph{Acknowledgement}
We are grateful to A. Munemasa for his help to the Groebner basis calculation in this paper. Also we would like to thank A.~Barg, O.~Musin and W-H.~Yu for the fruitful discussion. E.~Bannai is supported in part by NSFC grant No. 11271257. T.~Okuda is supported by Grant-in-Aid for JSPS Fellow No.25-6095. 
\section{Appendix I}
In the table below, $n$ denotes the dimension, $t$ the harmonic index, the inner value is the corresponding $b_{n,t}$.
\begin{center}
\begin{tabular}{c|c|c|c|c|c|c|c|c}

&$n=3$&$n=4$&$n=5$&$n=6$ &$n=7$&$n=8$&$n=9$& $n=10$\\ \hline
$t=4$&3.33..&5&7&9.33..& 12&15&18.33..& 22\\ \hline
$t=6$&3.41..&5.29..&7.69..&10.69..&14.33..&18.67..&23.76..& 29.68..\\ \hline
$t=8$&3.44..&5.41..&8.01..&11.35..&15.55..&20.72..&27.004..& 34.52..\\ \hline
$t=10$&3.45..& 5.47..&8.18..&11.73..&16.26..&21.97..&29.04..& 37.69..\\ \hline
$t=12$&3.46..&5.51..&8.28..&11.95..&16.71..&22.77..&30.39..& 39.84..\\ \hline
$t=14$&3.46..&5.53..&8.35..&12.10..&17.22..&23.32..&31.33..&41.37.. \\ \hline
$t=16$&3.47..&5.54..&8.39..&12.21..&17.37..&23.71..&32.01..& 42.48..\\ \hline
$t=18$&3.47..&5.56.. &8.42..&12.28..&17.37..&24.004..&32.51..& 43.32..\\ \hline
$t=20$&3.47..&5.56..&8.45..&12.34..&17.49..&24.22..&32.89..& 43.97..\\ 
\end{tabular}
\end{center}

\section{Appendix II}
\begin{eqnarray*}
P_1=&&s_1 - \frac{2065889}{90}s_2s_{10}s_{11}^3s_{12}^6 + 
        \frac{133819}{5}s_2s_{10}s_{11}^3s_{12}^4 -\frac{53179}{6}s_2s_{10}s_{11}^3s_{12}^2
        + \frac{8691}{10}s_2s_{10}s_{11}^3 \\
 &&- \frac{42787813648741}{1534950}s_2s_{10}s_{11}
        s_{12}^{14} + \frac{115559013011239}{1534950}s_2s_{10}s_{11}s_{12}^{12} - 
        \frac{7079980409537}{85275}s_2s_{10}s_{11}s_{12}^{10} +\\ 
   &&     \frac{12340902711704}{255825}s_2s_{10}s_{11}s_{12}^8 - 
        \frac{451353797911}{28425}s_2s_{10}s_{11}s_{12}^6 + 
        \frac{502755386489}{170550}s_2s_{10}s_{11}s_{12}^4\\  
     &&   -\frac{16159950307}{56850}s_2s_{10}s_{11}s_{12}^2 +
        \frac{620366243}{56850}s_2s_{10}s_{11}
\end{eqnarray*}
\begin{eqnarray*}
P_2=&&    s_2^2 + \frac{123226544609}{153495}s_{12}^{14} - \frac{12142761148}{5685}s_{12}^{12} + 
        \frac{13149854522}{5685}s_{12}^{10} - \frac{22430288734}{17055}s_{12}^8 + 
        \frac{7204390444}{17055}s_{12}^6\\
&& - \frac{434722732}{5685}s_{12}^4 + 
        \frac{13667541}{1895}s_{12}^2 - \frac{518644}{1895},
\end{eqnarray*}
\begin{eqnarray*}
P_3=  &&  s_4 - \frac{931}{90}s_6s_8s_9s_{10}s_{11}s_{12}^6 + 
        \frac{1169}{15}s_6s_8s_9s_{10}s_{11}s_{12}^4 - 
        \frac{723}{10}s_6s_8s_9s_{10}s_{11}s_{12}^2 +\frac{14}{5}s_6s_8s_9s_{10}s_{11}\\
      && -\frac{6570067}{108}s_6s_9^2s_{10}s_{11}^2
        s_{12}^7 + \frac{2032163}{30}s_6s_9^2s_{10}s_{11}^2s_{12}^5 - 
        \frac{393431}{18}s_6s_9^2s_{10}s_{11}^2s_{12}^3 +    \frac{43413}{20}s_6s_9^2s_{10}s_{11}^2s_{12}\\
&& - 
        \frac{5384771}{270}s_6s_9^2s_{10}s_{12}^7 + 
        \frac{3875621}{180}s_6s_9^2s_{10}s_{12}^5 - 
        \frac{1182097}{180}s_6s_9^2s_{10}s_{12}^3 + \frac{36059}{60}s_6s_9^2s_{10}s_{12}
       \\ && + \frac{356573}{90}s_6s_{10}s_{11}^2s_{12}^7 - 
        \frac{117439}{36}s_6s_{10}s_{11}^2s_{12}^5 + 
        \frac{27467}{60}s_6s_{10}s_{11}^2s_{12}^3 + \frac{1849}{60}s_6s_{10}s_{11}^2s_{12} \\ &&-
        \frac{669664592777551}{18419400}s_6s_{10}s_{12}^{15} + 
        \frac{33389002989181}{341100}s_6s_{10}s_{12}^{13} - 
        \frac{165129256285693}{1534950}s_6s_{10}s_{12}^{11} \\ &&+ 
        \frac{21245809013441}{341100}s_6s_{10}s_{12}^9 - 
        \frac{10457057636549}{511650}s_6s_{10}s_{12}^7 + 
        \frac{1291518147979}{341100}s_6s_{10}s_{12}^5 \\ &&- 
        \frac{248918696267}{682200}s_6s_{10}s_{12}^3 + 
        \frac{1591261873}{113700}s_6s_{10}s_{12},
\end{eqnarray*}
\begin{eqnarray*}
P_4= &&s_5 - \frac{21413}{90}s_6s_8s_9s_{12}^6 + \frac{12698}{45}s_6s_8s_9s_{12}^4 - 
        \frac{547}{6}s_6s_8s_9s_{12}^2 + \frac{157}{15}s_6s_8s_9 + 
        \frac{234370871}{2700}s_6s_9^2s_{11}^3s_{12}^7 \\ &&- 
        \frac{10332679}{108}s_6s_9^2s_{11}^3s_{12}^5 + 
        \frac{8818663}{300}s_6s_9^2s_{11}^3s_{12}^3 - 
        \frac{612916}{225}s_6s_9^2s_{11}^3s_{12} - 
        \frac{119641879}{2700}s_6s_9^2s_{11}s_{12}^7\\ && + 
        \frac{13122571}{270}s_6s_9^2s_{11}s_{12}^5 - 
        \frac{2260261}{150}s_6s_9^2s_{11}s_{12}^3 + 
        \frac{1280941}{900}s_6s_9^2s_{11}s_{12} - \frac{33137083}{900}s_6s_{11}^3s_{12}^7 \\ &&
        + \frac{3377731}{90}s_6s_{11}^3s_{12}^5 - \frac{1659631}{150}s_6s_{11}^3s_{12}^3 + 
        \frac{292957}{300}s_6s_{11}^3s_{12} + \frac{73245280332707}{92097000}s_6s_{11}s_{12}^{15}\\ &&
       - \frac{450846965863}{189500}s_6s_{11}s_{12}^{13} + \frac{10986176538638}{3837375}s_6s_{11}s_{12}^{11} - 
        \frac{2315729753414}{1279125}s_6s_{11}s_{12}^9\\ && + 
        \frac{1139869081667}{1705500}s_6s_{11}s_{12}^7  - 
        \frac{21377631109}{142125}s_6s_{11}s_{12}^5 + 
       \frac{66029795539}{3411000} s_6s_{11}s_{12}^3 - 
        \frac{605545711}{568500}s_6s_{11}s_{12},
\end{eqnarray*}
\begin{eqnarray*}
P_5 = && s_6^2 + s_9^2 + \frac{31477292476}{51165}s_{12}^{14} - \frac{28546636678}{17055}s_{12}^{12} + 
        \frac{31772393212}{17055}s_{12}^{10} - \frac{18667559008}{17055}s_{12}^8\\ && + 
        \frac{2075243891}{5685}s_{12}^6 
 - \frac{130374328}{1895}s_{12}^4 + \frac{12793807}{1895}s_{12}^2
        - \frac{502928}{1895},
\end{eqnarray*}
\begin{eqnarray*}
P_6 = && s_7 + \frac{80066}{225}s_8s_{10}s_{11}s_{12}^6 - 
        \frac{11473}{25}s_8s_{10}s_{11}s_{12}^4 + \frac{514}{3}s_8s_{10}s_{11}s_{12}^2 - 
        \frac{326}{25}s_8s_{10}s_{11}\\ &&
       + \frac{149163889}{1350}s_9^3s_{10}s_{11}^2s_{12}^7 - 
        \frac{166091023}{1350}s_9^3s_{10}s_{11}^2s_{12}^5 + 
        \frac{1981177}{50}s_9^3s_{10}s_{11}^2s_{12}^3 - 
        \frac{882044}{225}s_9^3s_{10}s_{11}^2s_{12}\\&& +
        \frac{31101917}{1350}s_9^3s_{10}s_{12}^7 -\frac{5639221}{225}s_9^3s_{10}s_{12}^5 + 
        \frac{1748518}{225}s_9^3s_{10}s_{12}^3 - \frac{110357}{150}s_9^3s_{10}s_{12} \\ && - 
        \frac{10474681}{270}s_9s_{10}s_{11}^2s_{12}^7+ 
        \frac{9832354}{225}s_9s_{10}s_{11}^2s_{12}^5 - 
        \frac{128839}{9}s_9s_{10}s_{11}^2s_{12}^3 + 
        \frac{215869}{150}s_9s_{10}s_{11}^2s_{12}\\&& + 
        \frac{350356120184819}{15349500}s_9s_{10}s_{12}^{15} - 
        \frac{17538932686289}{284250}s_9s_{10}s_{12}^{13} + 
        \frac{261228066120826}{3837375}s_9s_{10}s_{12}^{11}\\ && - 
        \frac{50587288180003}{1279125}s_9s_{10}s_{12}^9 + 
        \frac{33283484702111}{2558250}s_9s_{10}s_{12}^7 - 
        \frac{343022968073}{142125}s_9s_{10}s_{12}^5\\ && + 
        \frac{132249006263}{568500}s_9s_{10}s_{12}^3 - 
        \frac{845412947}{94750}s_9s_{10}s_{12},
\end{eqnarray*}@
\begin{eqnarray*}
P_7 = && s_8^2 - \frac{12103}{25}s_8s_9s_{11}s_{12}^7 + \frac{15771}{25}s_8s_9s_{11}s_{12}^5
        - \frac{3884}{15}s_8s_9s_{11}s_{12}^3 + \frac{782}{25}s_8s_9s_{11}s_{12} - 
        \frac{123823}{50}s_9^2s_{11}^2s_{12}^6 \\ && + \frac{111398}{45}s_9^2s_{11}^2s_{12}^4 - 
        \frac{47836}{75}s_9^2s_{11}^2s_{12}^2 + \frac{5999}{150}s_9^2s_{11}^2 - 
        \frac{4958}{75}s_9^2s_{12}^6 + \frac{610697}{3150}s_9^2s_{12}^4 - 
        \frac{19913}{150}s_9^2s_{12}^2 \\ &&
        + \frac{25127}{1050}s_9^2 + 
        \frac{222257}{225}s_{11}^2s_{12}^6 - \frac{154501}{150}s_{11}^2s_{12}^4 + 
        \frac{1771}{6}s_{11}^2s_{12}^2 - \frac{1147}{50}s_{11}^2 - 
        \frac{7439110878101}{15349500}s_{12}^{14}\\ && + \frac{9949282989527}{7674750}s_{12}^{12} - 
        \frac{1807447305433}{1279125}s_{12}^{10} + \frac{2072199005099}{2558250}s_{12}^8 - 
        \frac{784671190763}{2984625}s_{12}^6\\ && + \frac{287748973489}{5969250}s_{12}^4 - 
        \frac{6099957393}{1326500}s_{12}^2  +\frac{343461673}{1989750},
\end{eqnarray*}
\begin{eqnarray*}
P_8= && s_8s_9^2 - \frac{931}{30}s_8s_{12}^6 + \frac{581}{15}s_8s_{12}^4 - \frac{27}{2}s_8s_{12}^2 
        + \frac{4}{5}s_8 + \frac{20417761}{300}s_9^3s_{11}^3s_{12}^7 - 
        \frac{13248179}{180}s_9^3s_{11}^3s_{12}^5 \\&& + \frac{6730367}{300}s_9^3s_{11}^3s_{12}^3 
        - \frac{310181}{150}s_9^3s_{11}^3s_{12} - \frac{87766301}{2700}s_9^3s_{11}s_{12}^7 +
        \frac{529613}{15}s_9^3s_{11}s_{12}^5 - \frac{4892887}{450}s_9^3s_{11}s_{12}^3\\ && + 
        \frac{305753}{300}s_9^3s_{11}s_{12}  - \frac{10569643}{900}s_9s_{11}^3s_{12}^7+ 
        \frac{1102409}{90}s_9s_{11}^3s_{12}^5 - \frac{89701}{25}s_9s_{11}^3s_{12}^3 + 
        \frac{97157}{300}s_9s_{11}^3s_{12}\\ && + \frac{190976459657431}{30699000}s_9s_{11}s_{12}^{15}
         - \frac{42777492988657}{2558250}s_9s_{11}s_{12}^{13} + 
        \frac{15659663136481}{852750}s_9s_{11}s_{12}^{11} \\&&- \frac{27213265067107}{2558250}s_9s_{11}s_{12}^9 + 
        \frac{5977648644193}{1705500}s_9s_{11}s_{12}^7 - \frac{280742045533}{426375}s_9s_{11}s_{12}^5  \\&&+ 
        \frac{74541803387}{1137000}s_9s_{11}s_{12}^3 - \frac{381858611}{142125}s_9s_{11}s_{12},
\end{eqnarray*}
\begin{eqnarray*}
P_9= && s_8s_{11}^2 - \frac{266}{9}s_8s_{12}^6 + 39s_8s_{12}^4 - \frac{43}{3}s_8s_{12}^2 + 
        s_8 + \frac{16326947}{135}s_9^3s_{11}^3s_{12}^7 - 
        \frac{2353435}{18}s_9^3s_{11}^3s_{12}^5 \\&& + \frac{3603841}{90}s_9^3s_{11}^3s_{12}^3  - 
        \frac{111427}{30}s_9^3s_{11}^3s_{12} - \frac{4447387}{90}s_9^3s_{11}s_{12}^7 + 
        \frac{2887003}{54}s_9^3s_{11}s_{12}^5 \\ && - \frac{1473799}{90}s_9^3s_{11}s_{12}^3 + 
        \frac{68422}{45}s_9^3s_{11}s_{12}  - \frac{4821649}{270}s_9s_{11}^3s_{12}^7 + 
        \frac{337631}{18}s_9s_{11}^3s_{12}^5\\&& - \frac{497281}{90}s_9s_{11}^3s_{12}^3 + 
        \frac{2507}{5}s_9s_{11}^3s_{12} + \frac{10917806127109}{767475}s_9s_{11}s_{12}^{15} - 
        \frac{176502793283831}{4604850}s_9s_{11}s_{12}^{13} \\ &&+ 
        \frac{32394402183899}{767475}s_9s_{11}s_{12}^{11} - 
        \frac{37638196267127}{1534950}s_9s_{11}s_{12}^9 + 
        \frac{1380774469283}{170550}s_9s_{11}s_{12}^7 \\ &&- 
        \frac{777290373751}{511650}s_9s_{11}s_{12}^5 + 
        \frac{12800538154}{85275}s_9s_{11}s_{12}^3 - 
        \frac{1029239317}{170550}s_9s_{11}s_{12},
\end{eqnarray*}
\begin{eqnarray*}
P_{10}=  && s_8s_{12}^8 - \frac{204}{133}s_8s_{12}^6 +\frac{732}{931}s_8s_{12}^4 - 
        \frac{144}{931}s_8s_{12}^2 + \frac{9}{931}s_8 - \frac{1961}{10}s_9^3s_{11}^3s_{12}^7 + 
        \frac{55473}{266}s_9^3s_{11}^3s_{12}^5 \\ && - \frac{591987}{9310}s_9^3s_{11}^3s_{12}^3 + 
        \frac{27351}{4655}s_9^3s_{11}^3s_{12} + \frac{2467}{30}s_9^3s_{11}s_{12}^7 - 
        \frac{11731}{133}s_9^3s_{11}s_{12}^5 + \frac{125099}{4655}s_9^3s_{11}s_{12}^3 \\ &&- 
        \frac{22983}{9310}s_9^3s_{11}s_{12} + \frac{451}{10}s_9s_{11}^3s_{12}^7 - 
        \frac{5968}{133}s_9s_{11}^3s_{12}^5 + \frac{60906}{4655}s_9s_{11}^3s_{12}^3 - 
        \frac{1551}{1330}s_9s_{11}^3s_{12} \\ &&- \frac{14205863}{900}s_9s_{11}s_{12}^{15} + 
        \frac{9613783}{225}s_9s_{11}s_{12}^{13} - \frac{67333066}{1425}s_9s_{11}s_{12}^{11} + 
        \frac{641577172}{23275}s_9s_{11}s_{12}^9 \\ && -\frac{60864437}{6650}s_9s_{11}s_{12}^7 + 
        \frac{40366983}{23275}s_9s_{11}s_{12}^5 - \frac{16224543}{93100}s_9s_{11}s_{12}^3 + 
        \frac{167133}{23275}s_9s_{11}s_{12},
\end{eqnarray*}
\begin{eqnarray*}
 P_{11}= &&  s_9^4 - \frac{931}{3}s_9^2s_{12}^6 + \frac{1036}{3}s_9^2s_{12}^4 - 111s_9^2s_{12}^2 +
        10s_9^2  - \frac{1922475898}{17055}s_{12}^{14} + \frac{31186458317}{102330}s_{12}^{12} \\ &&- 
        \frac{5734068788}{17055}s_{12}^{10} + \frac{3329333917}{17055}s_{12}^8 - 
        \frac{364230059}{5685}s_{12}^6 + \frac{67069961}{5685}s_{12}^4 - \frac{2116188}{1895}s_{12}^2 + 
        \frac{156859}{3790},
\end{eqnarray*}
\begin{eqnarray*}
P_{12}= && s_9^2s_{12}^8 - \frac{186}{133}s_9^2s_{12}^6 + \frac{33}{49}s_9^2s_{12}^4 - 
        \frac{18}{133}s_9^2s_{12}^2 + \frac{9}{931}s_9^2 - \frac{931}{30}s_{12}^{14} + 
        \frac{1232}{15}s_{12}^{12} - \frac{8414}{95}s_{12}^{10} \\&&+ \frac{33227}{665}s_{12}^8 - 
        \frac{73329}{4655}s_{12}^6 + \frac{12756}{4655}s_{12}^4 - \frac{117}{490}s_{12}^2 +\frac{36}{4655},
\end{eqnarray*}
\[P_{13}=   s_{10}^2 + s_{11}^2 + s_{12}^2 - 1,\]
\begin{eqnarray*}
P_{14}= &&  s_{11}^4 +\frac{266}{3} s_{11}^2s_{12}^6 - \frac{334}{3}s_{11}^2s_{12}^4 + 44s_{11}^2s_{12}^2
        - 6s_{11}^2 + \frac{5261982208}{30699}s_{12}^{14} - \frac{4743342590}{10233}s_{12}^{12} \\ &&+ 
        \frac{5243879872}{10233}s_{12}^{10} - \frac{1019857462}{3411}s_{12}^8 + 
        \frac{337961738}{3411}s_{12}^6 - \frac{21126893}{1137}s_{12}^4 + \frac{2064566}{1137}s_{12}^2 - 
        \frac{26734}{379},
\end{eqnarray*}
\begin{eqnarray*}
P_{15}= &&  s_{11}^2s_{12}^8 - \frac{186}{133}s_{11}^2s_{12}^6 + \frac{33}{49}s_{11}^2s_{12}^4 - 
        \frac{18}{133}s_{11}^2s_{12}^2 + \frac{9}{931}s_{11}^2 - \frac{266}{9}s_{12}^{14} + \frac{241}{3}s_{12}^{12}
        - \frac{35423}{399}s_{12}^{10} + \frac{47770}{931}s_{12}^8 \\ &&- \frac{15469}{931}s_{12}^6 + 
        \frac{2784}{931}s_{12}^4 - \frac{255}{931}s_{12}^2 + \frac{9}{931},
\end{eqnarray*} 
\begin{eqnarray*}
P_{16}= &&   s_{12}^{16} - \frac{390}{133}s_{12}^{14} + \frac{63765}{17689}s_{12}^{12} - \frac{299970}{123823}s_{12}^{10} 
        + \frac{843138}{866761}s_{12}^8 - \frac{207090}{866761}s_{12}^6 + \frac{30375}{866761}s_{12}^4 -
        \frac{2430}{866761}s_{12}^2 \\ &&+ \frac{81}{866761}
\end{eqnarray*}


\begin{thebibliography}{99}
\bibitem{AAR} G.~Andrews, R.~Askey and R.~Roy, Special function, \textit{Encyclopedia of Mathematics and Its Applications, Cambridge University Press}.1999.
\bibitem{BB99}
Ei.~Bannai and  Et.~Bannai, \textit{Algebraic Combinatorics on Spheres}, (in Japanese) Springer Tokyo 1999.
\bibitem{BB09} Ei.~ Bannai and Et~. Bannai, A survey on spherical designs and algebraic combinatorics on spheres, \textit{European J. Combin.} \textbf{30} (2009), 1392--1425.
\bibitem{BH} A.~Brouwer and W.~Haemers, \textit{Spectra of graphs}, Universitext, Springer, New York, 2012.
\bibitem{BRV} A.~Bondarenko, D.~Radchenko and M.~Viazovska, 
Optimal asymptotic bounds for spherical designs, 
\textit{Ann.~of Math.} \textbf{178} (2013), 1--10.
\bibitem{BRV2} A.~Bondarenko, D.~Radchenko and M.~Viazovska, 
Well separated spherical designs, 
preprint, 25 pp, arXiv:1303.5991.
\bibitem{D} P.~Delsarte, An algebraic approach to the association schemes of coding theory,
\textit{Philips Res. Rep. Suppl.} \textbf{10} (1973).
\bibitem{DGS} P.~Delsarte, J.M.~Goethals, and J.J.~Seidel, Spherical codes and designs,
\textit{Geom. Dedicata} \textbf{6} (1977), 363--388.
\bibitem{ES} S.J.~ Einhorn, I.J.~ Schoenberg, On Euclidean sets having only two distances between points. I. II, \textit{Nederl. Akad. Wetensch. Proc. Ser. A Indag. Math.} \textbf{28} (1966) 479--488., 489--504.
\bibitem{E} A.~Erdelyi et al., \textit{Higher Transcendental Function}, Vol II, (Bateman Manuscript Project), MacGraw-Hill, 1953. 
\bibitem{HP} P.~de la Harpe and C.~Pache, Cubature formulas, geometrical designs, reproducing kernels, and Markov operators, \textit{Infinite groups: geometric, combinatorial and dynamical aspects, Progr. Math.}, \textbf{248}, Birkhauser, Basel, 2005, pp.,219--267. 
\bibitem{Hong} Y.~Hong, On spherical $t$-designs in $\R^n$, \textit{European J. Combin.} \textbf{3} (1982) 255--258.
\bibitem{LRS} D.~G.~Larman, C.~A.~Rogers, and J.~J.~ Seidel, On two-distance sets in Euclidean space, \textit{Bull. London Math. Soc.} \textbf{9} (1977), 26--267.
\bibitem{Lemmens} P.~W.~H.~Lemmens and J.~J.~Seidel, Equiangular lines, \textit{J. Alg} \textbf{24}(1973) 494--512.
\bibitem{Musin} O.~ Musin, Spherical two-distance sets, \textit{J.Combin. Theory Ser. A} \textbf{116}(2009), no. 4, 988-995. 
\bibitem{SZ} P.~Seymour and T.~Zaslavsky, Averaging sets: a generalization of mean values and spherical designs, \textit{Adv. in Math.} \textbf{52} (1984), no. 3, 213--240. 
\bibitem{Szego} G.~Szeg\"o, \textit{Orthogonal Polynomials}, 4th edition, Amer. Math. Soc., Colloquim Publications, volume 23 (2003).
\end{thebibliography}
\end{document}